\documentclass[leqno,12pt]{article}
\usepackage{amssymb,latexsym,amsmath,amsbsy,amsthm,amsxtra,amsgen}
\usepackage{graphicx}
\usepackage{color}
\usepackage[all]{xy}
\usepackage{epsfig}
\usepackage{graphics}
\newfont{\msbm}{msbm10 at 11pt}
\newcommand {\R} {\mbox{\msbm R}}
\newcommand {\Z} {\mbox{\msbm Z}}

\newcommand {\N} {\mbox{\msbm N}}

\newcommand {\E} {\mbox{\msbm E}}

\newfont{\msbmsm}{msbm10 at 8pt}

\newcommand {\Zsm} {\mbox{\msbmsm Z}}

\def\P{\mathbb{P}}
\def\E{\mathbb{E}}
\def\eps{\varepsilon}
\def\ep{\epsilon}

\def\cV{\mathcal{V}}

\def\cT{\mathcal{T}}

\def\cG{\mathcal{G}}
\def\cF{\mathcal{F}}
\def\cE{\mathcal{E}}
\def\cD{\mathcal{D}}
\def\cC{\mathcal{C}}
\def\cB{\mathcal{B}}
\def\cA{\mathcal{A}}

\newcommand{\indic}[1]{\mathbf{1}_{\{#1\}}}

\def\be{\begin{equation}}
\def\ee{\end{equation}}
\def\ba{\begin{align}}
\def\ea{\end{align}}

\def\Var{\textup{Var}}

\newtheorem{Theo}{Theorem}
\newtheorem{Lemma}[Theo]{Lemma}
\newtheorem{Cor}[Theo]{Corollary}
\newtheorem{Prop}[Theo]{Proposition}

\newtheorem{Rmk}[Theo]{Remark}

\begin{document}
\title{Survival of near-critical branching Brownian motion}
{\author{by Julien Berestycki\thanks{Supported by the \emph{Agence Nationale de la
Recherche} grants ANR-08-BLAN-0220-01 and ANR-08-BLAN-0190.}, Nathana\"el Berestycki\thanks{Supported by EPSRC grant EP/G055068/1}, and Jason Schweinsberg\thanks{Supported in part by NSF Grant DMS-0805472} \\
{\normalsize Universit\'e Paris VI, University of Cambridge, and University of California at San Diego}}
\maketitle

\begin{abstract}
Consider a system of particles performing branching Brownian motion with negative drift $\mu = \sqrt{2 - \eps}$ and killed upon hitting zero. Initially there is one particle at $x>0$. Kesten \cite{kesten} showed that the process survives with positive probability if and only if $\eps>0$. Here we are interested in the asymptotics as $\eps\to 0$ of the survival probability $Q_\mu(x)$. It is proved that if $L= \pi/\sqrt{\eps}$ then for all $x \in \R$, $\lim_{\eps \to 0} Q_\mu(L+x) = \theta(x) \in (0,1)$ exists and is a travelling wave solution of the Fisher-KPP equation. Furthermore, we obtain sharp asymptotics of the survival probability when $x<L$ and $L-x \to \infty$. The proofs rely on probabilistic methods developed by the authors in \cite{bbs}. This completes earlier work by Harris, Harris and Kyprianou \cite{hhk06} and confirms predictions made by Derrida and Simon \cite{ds07}, which were obtained using nonrigorous PDE methods.
\end{abstract}
\newpage

\section{Introduction}

\subsection{Main results}
Consider branching Brownian motion started at $x > 0$, in which each particle splits into two at rate one, drifts to the left at rate $\mu$, and is killed upon reaching the origin.  Kesten \cite{kesten} showed that the process dies out almost surely if $\mu \geq \sqrt{2}$ and survives forever with positive probability if $\mu < \sqrt{2}$.  We consider here the probability $Q_{\mu}(x)$ of survival when $\mu < \sqrt{2}$.  Let $\eps  = {2} - \mu^2 > 0$, and choose $L$ such that $1 - \mu^2/2 - \pi^2/2L^2 = 0$.  That is, we have $L = \pi/\sqrt{\eps}$.
The main objective of this paper is to prove the following results concerning the asymptotics of $Q_{\mu}(x)$ as $\eps \rightarrow 0$.

\begin{Theo}\label{T:kpplim}
There is a function $\theta: \R \rightarrow (0,1)$ such that $$\lim_{\eps \rightarrow 0} Q_{\mu}(L + \alpha) = \theta(\alpha)$$ for all $\alpha \in \R$.  The function $\theta$ satisfies the differential equation
\begin{equation}\label{kolmeq}
\frac{1}{2} \theta''  = \sqrt{2} \theta' - \theta(1 - \theta)
\end{equation}
with the boundary conditions $\lim_{\alpha \rightarrow \infty} \theta(\alpha) = 1$ and $\lim_{\alpha \rightarrow -\infty} \theta(\alpha) = 0$.
\end{Theo}

Theorem \ref{T:kpplim} does not fully determine the value of $\theta(\alpha)$ because solutions to (\ref{kolmeq}) are unique only up to a translation.  However, from Theorem \ref{T:kpplim} and the fact that $x \mapsto Q_{\mu}(x)$ is increasing, we immediately obtain the following corollary.
\begin{Cor}
Let $f:(0, \infty) \rightarrow \R$ be a function such that $\lim_{\eps \rightarrow 0} f(\eps) = \infty$.  Then $$\lim_{\eps \rightarrow 0} Q_{\mu}(L + f(\eps)) = 1$$ and $$\lim_{\eps \rightarrow 0} Q_{\mu}(L - f(\eps)) = 0.$$
\end{Cor}

The next theorem establishes more precise asymptotics for the survival probability when $x$ is much smaller than $L$, in which case the probability of survival tends to zero.  For this result, we allow $x$ to be a function of $\eps$, as long as $L - x \rightarrow \infty$ as $\eps \rightarrow 0$.  The result therefore applies, for example, when $x$ is a fixed number, or when $x = L^{\alpha}$ for $0 < \alpha < 1$, or when $x = L - \log L$.

\begin{Theo}\label{T:survivalthm}
There exists a constant $C$ such that if $L - x \rightarrow \infty$ as $\eps \rightarrow 0$, then $$Q_{\mu}(x) \sim C L e^{-\mu (L - x)} \sin \bigg( \frac{\pi x}{L} \bigg),$$
where $\sim$ means that the ratio of the two sides tends to one as $\eps \rightarrow 0$.
\end{Theo}

Finally, we present a result which shows that, if initially there is one particle at $x=L+\alpha$, where $\alpha \in \R$, so that the probability that the process survives forever is bounded between 0 and 1, then the descendant particles quickly settle to the stable configurations discussed in \cite{bbs}. These are precisely the configurations that are needed to apply Proposition 1 and Theorem 2 in \cite{bbs}.

Let $Z_\eps(t) = \sum_{i=1}^{M_\eps(t)} e^{\mu X_i(t)} \sin( \pi X_i(t)/L) \indic{X_i(t) \leq L}$, where $(X_i(t))_{1 \le i \le M_\eps(t)}$ denotes the set of active particles at time $t$. Likewise, denote
$Y_\eps(t) = \sum_{i=1}^{M_\eps(t)} e^{\mu X_i(t)} $.

\begin{Prop}\label{P:ic} Let $t =  c L^2$, where $c>0$ is arbitrary. Then if initially there is one particle at $x = L+\alpha$, there exists a nonnegative random variable $W$, with Laplace transform given by \eqref{LapTW}, such that
\begin{equation}\label{E:icZ}
\frac{Z_\eps(t)}{\eps^{1/2} \exp({\pi \sqrt{2} \eps^{-1/2} }) } \to  2\pi^2e^{\sqrt{2}\alpha} {W}.
\end{equation}
in distribution, while
\begin{equation}\label{E:icY}
\frac{Y_\eps(t)}{\exp({\pi \sqrt{2} \eps^{-1/2} }) }  \to 0
\end{equation}
in probability.
\end{Prop}

As an immediate consequence of Proposition \ref{P:ic} and \cite[Theorem 2]{bbs}, we obtain the following corollary. Let $n\ge 1$, and let $T>0$ be fixed.
Consider the coalescent process defined as follows.  Choose $n$ particles uniformly at random from the population at time $\eps^{-3/2} T$, and label these particles at random by the integers $1, \dots, n$.  For $0 \leq s \leq  T$, define $\Pi_\eps(s)$ to be the partition of $\{1, \dots, n\}$ such that $i$ and $j$ are in the same block of $\Pi_\eps(s)$ if and only if the particles labeled $i$ and $j$ are descended from the same ancestor at time $(T - s)\eps^{-3/2}$. Let $(\Pi(s), s \ge 0)$ denote the Bolthausen-Sznitman coalescent restricted to $n$ particles. See, e.g., \cite{bbs} for more precise definitions and \cite{ensaios} for background on coalescence.

\begin{Cor}
The processes $(\Pi_\eps(s), 0 \le s <T)$ converge as $\eps \to 0$, in the sense of finite-dimensional distributions, to the Bolthausen-Sznitman coalescent $\left(\Pi\left( \frac{ s}{\pi^2 \sqrt{2}} \right), 0 \leq s <T\right)$.
\end{Cor}

\begin{Rmk}
  Note that in the above result it is essential to restrict to $s<T$, as for $s=T$, the partition $\Pi_\eps(T)$ is the trivial partition consisting of exactly one block. (Indeed, by construction all particles are descended from the same individual at time 0).
\end{Rmk}

\subsection{Ideas behind the proofs}

The proofs of Theorems \ref{T:kpplim} and \ref{T:survivalthm} depend heavily on results in \cite{bbs}, of which this paper is a sequel.  In \cite{bbs}, we chose a different parameterization.  More precisely, for each $N \in \N$ we let $L = (\log N + 3 \log \log N)/\sqrt{2}$, which means that $\eps = \eps_N = 2\pi^2/(\log N + 3 \log \log N)^{2}$.  Consequently, obtaining asymptotic results as $\eps \rightarrow 0$ is equivalent to obtaining asymptotic results as $N \rightarrow \infty$.

The main result in \cite{bbs} can be described as follows. Let $M_N(t)$ be the number of particles alive at time $t$.
Denote the positions of the particles at time $t$ by $X_1(t) \geq X_2(t) \geq \dots X_{M_N(t)}(t)$.  Let
$$
Z_N(t) := \sum_{i=1}^{M_N(t)} e^{\mu X_i(t)} \sin \bigg(\frac{\pi X_i(t)}{L}\bigg) \mathbf{1}_{\{ X_i(t) \le L\}}, \qquad t\ge 0,
$$
and
$$
Y_N(t) := \sum_{i=1}^{M_N(t)} e^{\mu X_i(t)}  , \qquad t\ge 0.
$$
For each $N$, pick an initial configuration $X_1(0), \ldots, X_{M_N(0)}(0)$ such that $Z_N(0)/N(\log N)^2$ converges in distribution to some nondegenerate random variable $W$ as $N \rightarrow \infty$ and such that $Y_N(0) =o(N(\log N)^3).$ Then the processes $(Z_N(t),t\ge 0)$ converge in the sense of finite dimensional marginals to a limit $(Z(t), t\ge 0)$.  The limiting process $(Z(t), t\ge 0)$ is a continuous-state branching process (CSBP) with branching mechanism
$\psi(u) = au + 2 \pi^2 u \log u$, where $a \in \R$ is a constant whose value remains unknown.  This CSBP in the case $a = 0$ was introduced by Neveu in \cite{neveu}.  Therefore, for such a sequence of initial configurations, it is not surprising that the probability of extinction of the branching Brownian motion converges to the probability that $$\lim_{t \rightarrow \infty} Z(t) = 0$$ when $Z(0) = W$.  This probability is nontrivial whenever $W$ is not degenerate.

Unfortunately, we can not apply this result directly to a sequence of initial configurations which consists for each $N$ of a single particle at $L + \alpha$ because the condition on $Y_N(0)$ fails.  Instead we consider stopping the particles when they first hit a barrier at $L-y$, where $y \to \infty$ as $N\to \infty$ but $y \ll L$. We call $N_y$ the total number of particles that hit $L-y$, and we use as our initial configuration $N_y$ particles situated at $L-y.$  Because the process started from a single particle at $x$ becomes extinct if and only if the descendants of each of the $N_y$ particles at $L - y$ die off, this formulation is equivalent.  The upshot is that this new initial configuration falls in the application field of the results of \cite{bbs}, and we know explicitly the distribution of the random variable $W$.

The argument above allows us to show that for a given $\alpha$, the quantity $Q_{\mu}(L + \alpha)$ converges to a limit, say $\theta (\alpha)$, as $\eps \to 0.$  Once this is known, it is relatively straightforward, using what is known about the random variable $W$, to show that $\theta$ in fact solves (\ref{kolmeq}), which completes the proof of Theorem \ref{T:kpplim}.  To prove Theorem \ref{T:survivalthm}, we use results in \cite{bbs} to estimate the probability that a particle eventually reaches $L - \alpha$ and then apply the result of Theorem \ref{T:kpplim}.

\subsection{Related results and models}

Models of branching Brownian motion with absorption and branching random walk with absorption have received a lot of attention lately. We review here some of the pertinent known results.

Harris, Harris, and Kyprianou \cite{hhk06} showed (see Theorem 13) that the function $x \mapsto Q_{\mu}(x)$ satisfies Kolmogorov's equation $$\frac{1}{2} Q_{\mu}''(x) =  \mu Q_{\mu}'(x)  - Q_{\mu}(x)(1 - Q_{\mu}(x))$$
with boundary conditions $\lim_{x \rightarrow 0} Q_{\mu}(x) = 0$ and $\lim_{x \rightarrow \infty} Q_{\mu}(x) = 1$.  They also showed (see Theorem 1) that for each fixed $\mu < \sqrt{2}$ there is a constant $K$ such that $$\lim_{x \rightarrow \infty} e^{(\sqrt{\mu^2 + 2} - \mu)x} (1 - Q_{\mu}(x)) = K.$$  Note that the results in \cite{hhk06} are stated in terms of the extinction probability rather than the survival probability.

Simon and Derrida \cite{ds07, sd08} obtained quite precise estimates for the survival probabilities.  They considered a Brownian motion with diffusion coefficient $\sigma^2 = 2$, so their value of $L$ is twice as large as ours and their critical velocity is $2$ rather than $\sqrt{2}$.  However, translating their results into our context and notation, they obtain (see equation (B.16) of \cite{sd08}) that there is a constant $C$ such that when $L - x \gg 1$, we have
$$Q_{\mu}(x) = C L e^{\sqrt{2}(x - L)} \bigg( \sin \bigg( \frac{\pi x}{L} \bigg) + O \bigg( \frac{1}{L^2} \bigg) \bigg) + O(e^{2 \sqrt{2}(x - L)}).$$  See also equation (6) of \cite{ds07}.  They also show (see equation (B.17) of \cite{sd08})
that there is another constant $c$ such that when $x > L$ or $L - x$ is $O(1)$, we have $$Q_{\mu}(x) = 1 - \theta(x - L + c) + O \bigg( \frac{1}{L^2} \bigg),$$ where $\theta$ solves the differential equation (\ref{kolmeq}) above.  Note that these results would imply Theorems \ref{T:kpplim} and \ref{T:survivalthm}.  Although the derivations in \cite{ds07, sd08}, which are based on differential equations arguments, are not fully rigorous, we believe that it is probably possible to fill in the details in this argument and obtain the results by the methods used in \cite{ds07, sd08}.  We choose instead to pursue a probabilistic approach which shows the connections between survival probabilities and Neveu's CSBP.

Finally, while doing this project, we learned that  Aidekon, Harris, and Pemantle \cite{ahp} were working on similar questions using a spine-decomposition technique. 
We believe that their approach will yield the same results that we have.

There has also been a surge of recent work on survival probabilities for branching random walks in which particles are killed if they get to the left of a wall.  For results in this direction, see Gantert, Hu, and Shi \cite{ghs}, B\'erard and Gou\'er\'e \cite{bego09}, Jaffuel \cite{jaffuel}, and Feng and Zeitouni \cite{fz09}.

\section{Proof of Theorem \ref{T:kpplim}}

Throughout this section, let $(Z_t, t \geq 0)$ be a continuous-state branching process with branching mechanism $\psi(u) = au + 2 \pi^2 u \log u$.  We use $\P_x$ to denote probabilities when $Z_0 = x$ and $\P_{\nu}$ to denote probabilities when $Z_0$ has distribution $\nu$.  From results in \cite{grey}, we know that this process does not go extinct, meaning that if $Z_0 > 0$ then almost surely $Z_t > 0$ for all $t > 0$.  Nevertheless, we have $\lim_{t \rightarrow \infty} Z_t = 0$ with positive probability.  The proposition below relates this probability to the probability that branching Brownian motion goes extinct.

\begin{Prop}\label{P:lim_prob_ext}
Consider a branching Brownian motion with drift $\mu= \sqrt{2 - \eps_N}$. Let
$$
Z_N(t) = \sum_{i=1}^{M_N(t)} e^{\mu X_i(t)} \sin(\pi X_i(t) / L)
$$
and $$Y_N(t) = \sum_{i=1}^{M_N(t)} e^{\mu X_i(t)}.$$
Choose the initial configurations so that $Z_N(0) / N (\log N)^2 \to \nu$ in distribution as $N \to \infty$, and $Y_N(0)/(N(\log N)^3)$ converges in probability to 0.  Let $\cE$ denote the event that $\limsup_{t \to \infty} Z_t = 0$, and let $\cE_N = \{\lim_{t \to \infty} M_N(t) = 0\}$, the event of extinction of our branching Brownian motion.
Then as $N \to \infty$, $\lim_{N \to \infty} \P(\cE_N) = \P_\nu(\cE).$
\end{Prop}

The key idea of the proof is to use Proposition 1 in \cite{bbs}. However, some care is needed because the convergence in this proposition holds only in the sense of finite-dimensional distributions, and this can not be
extended to weak convergence in the Skorokhod topology. The argument will therefore roughly
consist of showing that if one observes $Z_t$ for some large but finite $t$, then one can guess whether the event $\cE$ holds with very small error probability. Since Proposition 1 in \cite{bbs} shows convergence of the finite dimensional distributions, it remains to show some \emph{a priori} estimates for survival or extinction when $Z_N(t)$ is respectively large or small. This is achieved via martingale arguments and a comparison with an ordinary (Galton-Watson) branching process.

\begin{Lemma}\label{L:extprobinf}
For all $x \in \R$, we have $\P_x(\cE) = e^{- x \alpha}$, where $\alpha = \exp(- a/2\pi^2)$. 
\end{Lemma}

\begin{proof}
  By the branching property, $f(x) = \P_x(\cE)$ satisfies $f(x+y)  = f(x) f(y)$. To show that $f$ is indeed of the above type, it suffices to show that $f(x)$ is not identically 1.
  This in turn follows from the fact that
  $$
  \psi'(0) = - \infty < 0
  $$
  and the discussion on p. 716 of Bertoin et al. \cite{befose}. The identification $\alpha = \exp(- a/2\pi^2)$ also comes from that same result (it is the largest root of $\psi(u) = 0$).
\end{proof}

Note that in particular, for any $0 < \delta < 1$ we can find $\ep$ and $A> \ep$ such that $\P_\ep(\cE) > 1- \delta$ while $\P_A(\cE) \le \delta$.

\begin{Lemma}\label{L:notstuck}
Almost surely $\lim_{t \to \infty} Z_t$ exists and is either 0 or $\infty$.  In particular, if $0<\ep<A$, then $\P_\nu(Z_t \in [\ep, A]) \to 0$ as $t \to \infty$.
\end{Lemma}

\begin{proof}
Observe first that by the previous lemma, for any $y>0$, if $\tau_y = \inf\{t \ge 0: Z_t \le y\}$, and if $0<a<b$ are arbitrary, then $\inf_{x \in (a,b)} \P_x(\tau_{a/2}< \infty) >0$.  Therefore, for each integer $m$, there is a constant $K_m$ such that
\begin{equation}\label{unifbound}
\inf_{x \in (2^m, 2^{m+1})} \P_x(\tau_{2^{m-1}}< K_m) >0.
\end{equation}
For each integer $m$, let $\tau_m^0 = 0$, and then for $k \geq 0$, inductively define $\tau_m^{k+1} = \inf\{t \geq \tau_m^k + K_m: Z_t \in (2^m, 2^{m+1})\}$.  Consider the event $\cA_m = \{\liminf_{t \to \infty} Z_t  \in (2^{m}, 2^{m+1})\}$ for $m \in \Z$. Then on the event $\cA_m$, we have $\tau^k_m < \infty$ for all $k \ge 0$. By
(\ref{unifbound}), between each $\tau_m^k$ and $\tau_m^{k+1}$, there is a probability uniformly bounded below that the CSBP falls below $2^{m-1}$, and these trials are independent by the strong Markov property. Thus, almost surely on $\cA_m$, $\liminf_{t \to \infty} Z_t \le 2^{m-1}$, which means $\P_x(\cA_m) = 0$. Therefore, $\P(\cup_{m \in \Zsm} \cA_m) = 0$, so almost surely, $\liminf_{t \to \infty} Z_t \in \{0, \infty\}$.  By a similar argument, almost surely, $\limsup_{t \to \infty} Z_t \in \{0, \infty\}$. By Lemma \ref{L:extprobinf}, one also discards the possibility that $\liminf_{t \to \infty} Z_t = 0$ and $\limsup_{t \to \infty} Z_t = \infty$ occur simultaneously. This proves the lemma.
\end{proof}

\begin{Lemma}\label{L:extprobinf_approx}  Choose $\delta>0$ and $\ep > 0$ so that $\P_\ep( \cE) \ge 1- \delta$. Then there exists $t_0$ such that if $t\ge t_0$, $|\P_\nu( Z_t < \ep) - \P_\nu(\cE)|\le 3 \delta$.
\end{Lemma}

\begin{proof}
  We write for all $A> \ep$
  $$
  \P_\nu(\cE) = \P_\nu(\cE; Z_t \le \ep) + \P_\nu(\cE; Z_t \in (\ep, A)) + \P_\nu(\cE; Z_t > A).
  $$
  Choosing $A>0$ large enough that $\P_A(\cE) \le \delta$,
  \begin{align*}
  |\P_\nu(\cE) - \P_\nu(Z_t \le \ep)|  &\le  \P_\nu(\cE^\complement; Z_t \le \ep)+ \P_\nu(\cE; Z_t \in (\ep, A)) + \P_\nu(\cE; Z_t > A)\\
  & \le \P_\ep(\cE^\complement) + \P_\nu(Z_t \in [\ep, A]) + \P_A(\cE)\\
  & \le 3 \delta
  \end{align*}
  for $t \ge t_0$, since the middle term becomes smaller than $\delta$ for $t$ large enough by Lemma \ref{L:notstuck}.
\end{proof}

We now come to the main technical part of the proof, which says that for the branching Brownian motion, if $Z_N(t)$ happens to be large at some point, then it becomes unlikely that the process will ever become extinct. This is achieved through a rough comparison with a supercritical branching process. First we need to slightly reformulate a result of \cite{bbs}.  Here we are working with branching Brownian motion with particles killed both at $0$ and at $L$, and $({\cal F}_t, t \geq 0)$ denotes the natural filtration of this process.
\begin{Lemma}
  \label{L:Prop1618}
Assume that $Y_N(0)/(N(\log N)^3)$ converges in probability to 0 as $N \rightarrow \infty$.  Let $R$ be the number of particles that hit $L$ between times $0$ and $(\log N)^3$.
Then for sufficiently large $N$, we have
  \begin{equation}\label{Prop1618Exp}
  \E(R| \cF_0) \ge \pi \frac{Z_N(0)}{N ( \log N)^2}.
  \end{equation}
Also, there is a constant $C$ such that
  \begin{equation}\label{Prop1618Var2}
  \E(R^2 | \cF_0) \le C \left(\frac{Z_N(0)}{N ( \log N)^2}\right)^2 + C \frac{Z_N(0)}{N(\log N)^2} + o_p(1),
  \end{equation}
where $o_p(1)$ denotes a term which tends to zero in probability as $N \rightarrow \infty$.  If there is a deterministic sequence $a_N \rightarrow 0$ such that $Y_N(0)/(N(\log N)^3) \leq a_N$ for all $N$, then the $o_p(1)$ can be replaced by an $o(1)$ term, which tends to zero uniformly as $N \rightarrow \infty$.
\end{Lemma}

\begin{proof}
To prove \eqref{Prop1618Exp}, we use Proposition 16 in \cite{bbs}, taking $A = 0$, $s = 1$, and $\theta = 1$.  Note that this lower bound can be seen either from the first conclusion of Proposition 16 in \cite{bbs} or just from equation (71) of \cite{bbs}.

To get \eqref{Prop1618Var2}, we use Proposition 18 in \cite{bbs}.  Here, we can not apply the result directly because the result requires $Z_N(0)/(N(\log N)^2)$ to be bounded by some constant $\ep^{-1/2}$ such that $\theta \ep^{-1/2} \leq 1$.  However, we just observe that Proposition 18 of \cite{bbs} comes from combining equations (72), (73), and (78) of \cite{bbs}, which lead to the bound (\ref{Prop1618Var2}) when $A = 0$, $s = 1$, and $\theta = 1$.  Note that the $o_p(1)$ term in (\ref{Prop1618Var2}) comes from the term in (78) of \cite{bbs} that involves $Y_N$.  This term becomes $o(1)$ when $Y_N(0)/(N(\log N)^3) \leq a_N$ for all $N$.
\end{proof}

We now introduce the Galton-Watson process to which we compare the branching Brownian motion.
Consider branching Brownian motion with drift $-\sqrt{2}$ started with one particle at $L$, in which particles are killed when they reach $L - y$.  Let $T_y$ be the extinction time for this process.  Because $T_y < \infty$ almost surely, there exists an increasing function $g: (0, \infty) \rightarrow (0, \infty)$ such that $\P(T_y > g(y)) \rightarrow 0$ as $y \rightarrow \infty$.
Now let $y = y_N$ be a sequence that tends to infinity slowly enough that $y=o(L)$ and $g(2y)(\sqrt{2} - \mu) \rightarrow 0$ as $N \rightarrow \infty$.

\begin{Lemma}\label{L:approxNy}
Consider branching Brownian motion with drift $-\mu$ started with a single particle at $L$.  Choose $y = y_N$ as above, and let $N_y$ be the number of particles that are killed if particles are killed upon reaching $L - y$.
As $N \to \infty$ we have
\begin{equation}\label{E:conv to W}
N_{y}  e^{-\sqrt 2 y} y \to_d W,
\end{equation}
where $\rightarrow_d$ denotes convergence in distribution.  Furthermore, there exists a universal constant $D$ such that
  \begin{equation}
  \liminf_{N \to \infty} \E( y e^{-\sqrt 2 y} N_y \wedge D ) > \frac{2}{ \pi^2 \sqrt{2}}.
  \end{equation}
\end{Lemma}

\begin{proof}
The variable $N_y$ is a functional of a branching Brownian motion with drift $-\mu$ started from one particle at $L$ but it is easy to see that one can instead work with a branching Brownian motion with drift $-\sqrt 2$.
For this process, let $Z_y$ be the number of particles that hit position $L-y$ if initially there is one particle at L and particles are killed upon reaching $L-y$. To offset the change of drift it suffices to replace the operation of killing at position $L-y$ by killing particles upon hitting an oblique line whose slope is exactly $\sqrt 2 -\mu$. Hence in this model (a branching Brownian motion with drift $-\sqrt 2$ started from one particle at $L$) we take $N_y$ to be the number of particles that at some time $t>0$ reach the line $L-y-(\sqrt 2 - \mu)t$
if particles that hit this line are immediately killed.

As shown in Section 5 of Neveu \cite{nev87}, the process $(Z_u)_{u \ge 0}$ is a continuous-time branching process, and furthermore there exists a random variable $W$ such that, almost surely
as $y \to \infty$,
\begin{equation}\label{convNeveu}
y e^{-\sqrt{2} y} Z_y \to W.
\end{equation}

We now compare the distribution of $N_y$ with that $Z_y$. Observe first that $N_y \ge Z_y$. It is the other direction that requires a few more arguments. Note for instance that $N_y$ may be infinite, although this will happen with exceedingly small probability.
We claim that if $y_N' = y_N + g(2y_N)(\sqrt{2} - \mu)$, then
\begin{equation}\label{E:NyZy}
\P(N_y \leq Z_{y'}) \rightarrow 1
\end{equation}
as $N \to \infty$. Because $y' - y = g(2y)(\sqrt{2} - \mu) \rightarrow 0$ as $N \rightarrow \infty$, this will immediately entail that $y e^{-\sqrt{2} y} N_y $ converges in distribution to $W.$

Define the event $\cC = \{ T_{y'} < g(y')  \}.$  Note that $y' \leq 2y$ for sufficiently large $N$.  Therefore, for sufficiently large $N$, on the event $\cC$, we have that $y+ T_{y'} (\sqrt 2 - \mu) \leq y +  g(2y)(\sqrt 2 - \mu) = y'$.  In this case,
we have that $N_y \le Z_{y'}$, since until $T_{y'}$ the vertical line $L-y'$ is to the left of the oblique line $L-y -(\sqrt 2 - \mu)t$ and hence
particles have more time to branch if we kill at $L-y'.$
Because $\P(\cC) \to 1$ as $N \rightarrow \infty$ by the definition of the function $g$, equation \eqref{E:NyZy} follows.

The second part of the lemma follows easily by the dominated convergence theorem and the fact that $W$ has infinite mean, which is a consequence of Proposition 27 in \cite{bbs}.
\end{proof}

Consider the first $N_y  \wedge D y^{-1} e^{\sqrt{2} y}$ particles that reach $L - y$.  Let ${\hat N}$ be the number of descendants of these particles that reach $L$ within a time $(\log N)^3$ after the parent particle reached $L - y$ (assuming particles are killed when they reach $L$).

\begin{Lemma}\label{L:EN>1}
Choosing $D$ as in Lemma \ref{L:approxNy}, there is a universal $C< \infty$ such that
\begin{equation}\label{EN>1}
\E(\hat N) \ge  2 \; ; \;  \E(\hat N^2 ) \le C
\end{equation}
for all $N$ large enough.
\end{Lemma}

\begin{proof}
We estimate the expectation and second moment of $\hat N$ using Lemma \ref{L:Prop1618}.
Note that, conditionally on $N_y$, $\hat N$ has the same law as the number of descendants of $N_y \wedge Dy^{-1} e^{\sqrt{2}y}$ particles started from $L-y$, that hit $L$ prior to time $(\log N)^3$ (i.e., the time at which the $N_y$ particles reach $L-y$ is irrelevant).
Note that for this process,
$$
Y_N(0) = (N_y \wedge D y^{-1} e^{\sqrt{2}y})  e^{\mu (L-y)} \le  D y^{-1} e^{\mu L} e^{y(\sqrt{2} - \mu)},
$$
so $Y_N(0)/(N(\log N)^3) \rightarrow 0$ as $N \rightarrow \infty$ and we can apply Lemma 8.
Thus, applying the result of Lemma \ref{L:Prop1618} and using the inequality $\sin(x)\ge  (2/\pi)x$ for $x \in [ 0, \pi/2]$,
\begin{align*}
\E(\hat N | N_y) & \ge   \frac{(N_y\wedge Dy^{-1} e^{\sqrt 2 y }) e^{\mu (L - y)}\sin(\pi(L-y)/L) \pi }{N ( \log N)^2} \nonumber \\
&\sim \pi^2 \sqrt{2} (y e^{-\sqrt{2} y} N_y \wedge D),
\end{align*}
where $\sim$ means that the ratio of the two sides tends to one as $N \rightarrow \infty$.  Taking expectations of both sides and applying Lemma \ref{L:approxNy}, we get that
$$
\E(\hat N) \ge 2
$$
for $N$ large enough.

The second moment of $\hat N$ is controlled in a similar fashion.  By Lemma \ref{L:Prop1618},
$$
\E(\hat N^2) \le C  \left(\frac{D y^{-1} e^{\sqrt{2} y} e^{\mu(L-y)} \sin (\pi y /L)}{N ( \log N)^2}\right)^2  + C  \frac{D y^{-1} e^{\sqrt{2} y} e^{\mu(L-y)} \sin (\pi y /L)}{N ( \log N)^2} + o(1),
$$
which is bounded by a constant for sufficiently large $N$.  Equation \eqref{EN>1} is proved.
\end{proof}

For any particle that reaches $L$, we can associate a random variable $\hat N$ as above, which counts the number of offspring of that particle that first hit $L-y$ and then return to $L$ after no more than time $(\log N)^3$.
This gives rise to an (ordinary) Galton-Watson branching process $\cT$ whose offspring distribution is
the distribution of $\hat N$. Note that if this Galton-Watson process survives, then our branching Brownian motion cannot become extinct. To see this, it suffices to check that on the event of survival
there are particles alive at all times. Fix a ray $\xi \in \cT$. Let $T_i$ be the time at which
the $i^\text{th}$ particle of this ray hits $L$. Then $(T_{i+1} - T_i)_{i\ge 1}$ is stochastically bounded below by an i.i.d. sequence of random variables which are strictly positive (being the time for a Brownian motion with drift $\mu$ to hit $L-y$ started from $L$). Thus by the law of large numbers, $\lim_{i \to \infty} T_i = \infty$ almost surely, and hence there are particles alive at all times.
With this in mind, we can prove Lemma \ref{L:compGW} below.  The key tool in the proof will be a second-moment argument that gives a lower bound on the survival probability of a Galton-Watson process.

\begin{Lemma}\label{GWlem}
Let $(p_k)_{k=0}^{\infty}$ be a sequence of nonnegative numbers that sum to 1, and let $X$ be a random variable such that $\P(X = k) = p_k$ for all nonnegative integers $k$.  Let $q$ be the extinction probability of a Galton-Watson process started with a single individual with offspring distribution $p_k$.  Then $$1 - q \geq \frac{2(\E[X] - 1)}{\E[X(X-1)]}.$$
\end{Lemma}

\begin{proof}
Let $m = \E[X]$ and $\alpha = \E[X(X-1)]$.  Let $g(s)= \E(s^{X})$ be the generating function of the offspring distribution.
Then by differentiating under the expectation sign, the derivatives of all orders of $g$ are nonnegative, hence in particular $g''(s)$ is nondecreasing and $g''(s) \le g''(1) = \alpha$. Therefore, integrating between $s$ and $1$,
$$
g'(1) - g'(s) =\int_s^1 g''(t) dt \le \alpha(1-s)
$$
and hence
$
g'(s) \ge m - \alpha(1-s).
$
Integrating further gives
$$
g(1) - g(s) = \int_s^1 g'(t) dt \ge m(1-s) - \alpha\int_s^1 (1-t)dt = m(1-s) - \alpha (1-s)^2/2.
$$
Thus
$$
g(s) \le 1-m(1-s) + \alpha (1-s)^2/2 =: h(s).
$$
Thus, the extinction probability $q$ is smaller than the unique $s < 1$ such that the right-hand side $h(s) = s$. Writing $x = 1-s$, this means we are looking for $x>0$ such that $\alpha x^2/2- m x + x = 0$, or $x(1- m + \alpha x/2) = 0$, or $x = 2(m-1)/\alpha$.  The result follows.
\end{proof}

\begin{Lemma}
  \label{L:compGW}
  Assume there is initially a particle at $L$. Then there exists a number $\rho>0$, independent of $N$, such that
  $\P(\cE_N^\complement) \ge \rho$.
\end{Lemma}

\begin{proof}
By the above remark, it suffices to show that the survival probability for the Galton-Watson tree $\cT$, if there is initially one particle at $L$, is bounded away from 0.  By \eqref{EN>1}, we know that $\E[{\hat N}] \geq 2$ and $\E[{\hat N}({\hat N-1})] \leq \E[{\hat N}^2] \leq C$ for sufficiently large $N$.  It thus follows from Lemma \ref{GWlem} that the survival probability is at least
$2/C$ for sufficiently large $N$.  Furthermore, on the event that $\cT$ does not become extinct, we know that $\cE_N^\complement$
holds almost surely. This completes the proof of Lemma \ref{L:compGW}.
\end{proof}

It is now fairly simple to prove the following lemma.
\begin{Lemma}
  \label{L:Zlarge}
Consider branching Brownian motion with drift $\mu$ whose initial configurations satisfy the conditions of Proposition
\ref{P:lim_prob_ext} for a nonzero measure $\nu$.
 Let $\delta>0$. Then there exists $A>0$ such that for all $t>0$ and all large enough $N$,
  \begin{equation}
\P\left(\cE_N \left| \frac{Z_N(t(\log N)^3)}{N ( \log N)^2} > A  \right. \right) \le \delta.
  \end{equation}
\end{Lemma}

\begin{proof}
Let $t>0$. Choose $K$ such that $(1- \rho)^K \le \delta/2$, and let $\cB$ be the event that at least $K$ particles ever reach $L$. Define the event $\cA$ by
\begin{equation}\label{A}
\cA_N: = \{Z_{N}(t (\log N)^3) /N ( \log N)^2 > A\}
\end{equation}
for some $A>0$ to be defined later. Intuitively, we want to take $A>0$ large enough that $\P(\cB|\cA_N)>1-\delta/2$, for then we know by Lemma \ref{L:compGW} that survival will occur with probability at least $1-\delta$. Let $R$ be the number of particles that hit $L$ between times $t ( \log N)^3$ and $(t+1) ( \log N)^3$. Note first that if $\cF = \cF_{ t( \log N)^3}$,
\begin{align*}
  \P(R\ge 1 | \cF) & \ge \frac{\E(R | \cF)^2}{\E(R^2 | \cF)}.
\end{align*}
Lemma 20 of \cite{bbs} establishes that $Y_N(t(\log N)^3)/(N(\log N)^3)$ converges to 0 in probability as $N \rightarrow \infty$.  Therefore, by applying Lemma 8 and writing $V = Z_N(t (\log N)^3)/ (N(\log N)^2)$, we get
\begin{align*}
  \P(R\ge 1 | \cF) & \ge \frac{\pi^2 V^2}{CV^2 + CV + \delta_N},
\end{align*}
where $\delta_N \rightarrow 0$ in probability as $N \rightarrow \infty$.  Let $\cG_N$ be the event that $\delta_N \leq 1$.
Choose $0 < 2\eta < \pi^2/(2C + 1)$.  If $A \geq 1$, then $\pi^2V^2/(CV^2 + CV) \geq \pi^2/(2C + 1)$ on the event $\cA_N \cap \cG_N$, so
$$
\P(R \ge 1 | \cA_N) \geq \frac{\E(2\eta \mathbf{1}_{\cA_N \cap \cG_N})}{\P(\cA_N)}.
$$
Because $P(\cG_N) \rightarrow 1$ as $N \rightarrow \infty$ but $\liminf_{N \rightarrow \infty} P(\cA_N) > 0$ by Theorem 2 of \cite{bbs}, it follows that $\P(R\ge 1 \vert \cA) \ge \eta$ for sufficiently large $N$.

Let $(B_1, B_2, \ldots)$ be a sequence of i.i.d. Bernoulli random variables with success probability $\eta$.  Choose $r\ge 1$ so that $\P(\sum_{i=1}^r B_i \ge K) \ge 1- \delta/2$.  Observe that there is a universal constant $D > 1$ such that the largest contribution to $Z_N$ from any given particle is at most $D N ( \log N)^2$.  Now choose $A = (D+1)r$.  Then the population of particles can be broken into at least $r = A/(D+1)$ groups of particles, corresponding to index sets $I_1, \ldots, I_r$, such that for all $1\le j \le r$, we have
$$
\sum_{i \in I_j} e^{\mu X_i(t ( \log N)^3)} \sin\left(\frac{\pi X_i(t ( \log N)^3)}{L}\right) \ge N ( \log N)^2.
$$
If $1\le j \le r$, let $R_j$ denote the number of descendants of a particle in $I_j$ at time $t ( \log N)^3$ that hit $L$ before time $(t+1) ( \log N)^3$.
Then it was shown above that for sufficiently large $N$, we have $\P(R_j\ge 1 | \cA) \ge \eta$ for all $1\le j \le r$. Moreover by the branching property, the variables $(R_1, \ldots, R_r)$ are conditionally independent given $\cF$.  This leads to the desired inequality
\begin{equation}\label{BlikelyifA}
\P(\cB | \cA) \ge 1- \delta/2.
\end{equation}
Now,
\begin{equation}\label{Zlarge1}
\P(\cE_N| \cA) \le \P(\cB^c | \cA) + \P(\cE_N| \cA \cap \cB).
\end{equation}
Note that the first term in the right-hand side is smaller than $\delta/2$ by \eqref{BlikelyifA}, while by the branching property, and by Lemma \ref{L:compGW},
$$
\P(\cE_N| \cA \cap \cB) \le (1- \rho)^K \le \delta/2
$$
by choice of $K$. This finishes the proof of Lemma \ref{L:Zlarge}.
\end{proof}

Conversely, we need a result which tells us that if $Z_N(t) < \ep N ( \log N)^2$ then the process is likely to die out. This is done in the following simple lemma.

\begin{Lemma} \label{L:ext_if_small}
Consider branching Brownian motion with drift $\mu$ whose initial configurations satisfy the conditions of Proposition
\ref{P:lim_prob_ext}.
For all $\delta>0$ there is $\ep>0$ such that for all $t>0$, if $\cD= \{Z_N(t( \log N)^3) < \ep N ( \log N)^2\}$, then $\P( \cE_N | \cD) \ge 1- \delta$.
\end{Lemma}

\begin{proof}
We may assume that $t=0$ and that all particles are initially to the left of $L-1$ (because if there were one particle to the right of $L-1$ then $Y_N(t ( \log N)^3)$ would be greater than $e^{\mu(L-1)}$, an event which has probability tending to zero by
Lemma 20 of \cite{bbs}.  Let $\bar Z(s) = \sum_{i=1}^{\bar M_N(s)} e^{\mu \bar X_i(s)} \sin ( \pi \bar X_i(s) / L)$, where $\bar X_i(s)$ is the position of the $i^{\text{th}}$ particle when particles are killed upon hitting $L$. Then $\bar Z$ is a nonnegative martingale, and thus converges almost surely (and is hence almost surely bounded). It is easy to see that this limit may only be zero, i.e., that all particles eventually die out, either by hitting $L$ or by hitting 0. (Indeed, otherwise, there is some $\eta>0$ such that the interval $[\eta, L-\eta]$ has particles at an unbounded set of times. An application of the strong Markov property then shows that the number of particles in this interval eventually exceeds any number, which contradicts the almost sure boundedness of $\bar Z$ as a function of time). We now claim that with overwhelming probability, all particles must die by hitting 0. Indeed, let $\tau_x = \inf\{t\ge 0: Z(t) \le x N ( \log N)^2 \}$ and let $\tau'_x = \inf\{t\ge 0: Z_t \ge x N( \log N)^2\}$. Since $\bar Z$ is a martingale which makes bounded jumps (so that $(\bar Z_{t \wedge \tau'_x \wedge \tau_0}, t\ge 0)$ is bounded), by the optional stopping theorem,
\begin{equation}\label{ost}
\P( \tau'_x < \tau_0 | \cD) \le \ep/x.
\end{equation}
Assume that at least one particle hits $L$. On this event, then an ancestor of this particle must have hit level $L-1$ since initially all particles are to the left of $L-1$.
But at that point, $\bar Z(s)$ is at least $e^{\mu(L-1)} \sin(\pi /L) \ge c N( \log N)^2$, where $c > 0$ is a positive constant. Therefore, on this event $\tau'_c < \tau_0$. Thus the probability of this event is, by \eqref{ost}, at most $\ep/c$. Thus it suffices to choose $\ep>0$ such that $\ep/c < \delta$ and the statement of the lemma holds.
\end{proof}

\begin{proof}[Proof of Proposition \ref{P:lim_prob_ext}] Let $\delta >0$. Choose $\ep>0$ small enough that
$\P_\ep(\cE) \ge 1- \delta$ and small enough that the conclusion of Lemma \ref{L:ext_if_small} holds.
Choose $A>0$ large enough that Lemma \ref{L:Zlarge} holds.
By Lemmas \ref{L:notstuck} and \ref{L:extprobinf_approx} we may now fix $t$ such that $|\P_\nu (\cE) - \P_\nu(Z_t < \ep)| \le 3\delta$ and also such that $\P_\nu(Z_t \in [\ep, A]) \le \delta$.

Note that, letting $V_N(t) = Z(t ( \log N)^3) / (N( \log N)^2)$,
\begin{align*}
\P(\cE_N) &= \P\left(\cE_N; V_N(t) < \ep\right) + \P\left(\cE_N; V_N(t) \in [\ep,A] \right)  + \P\left(\cE_N; V_N(t) > A\right)
\end{align*}
and thus
\begin{align}
|\P(\cE_N) - \P_\nu(\cE)| & \le |\P_\nu(\cE) - \P_\nu(Z_t < \ep) | + \left|\P\left(\cE_N; V_N(t) < \ep\right)  - \P_\nu(Z_t< \ep)\right|  \nonumber \\
& +  \P\left(\cE_N; V_N(t) \in [\ep,A] \right) + \P\left(\cE_N; V_N(t) > A\right).
\label{extinbd4}
\end{align}
We bound these four terms separately. By Lemma \ref{L:Zlarge} the fourth term is smaller than $\delta$, and the first one is smaller than $3\delta$ by choice of $t$. The third term is smaller than $2\delta$ for $N$ large enough: indeed by Proposition 1 in \cite{bbs}, $Z_N(t (\log N)^3)/N(\log N)^2 $ converges in distribution towards $Z_t$ started from $\nu$, and the set $[\ep, A]$ is closed.
Hence, the limsup of the probability in the third term is at most $\P_\nu(Z_t \in [\ep, A])$, which is at most $\delta$ by choice of $t$.

It remains to deal with the second term. We observe that on the one hand,
\begin{equation}\label{extinbd2}
\P\left(\cE_N; V_N(t) < \ep\right)  - \P_\nu(Z_t< \ep) \le \P\left( V_N(t) < \ep\right) - \P_\nu(Z_t < \ep)
\end{equation}
which converges to 0 again by Proposition 1 in \cite{bbs} and the fact that $\P_\nu(Z_t = \ep) =0$. Thus the left hand side is smaller or equal to $\delta$ for $N$ large enough. On the other hand, by choice of $\ep$ (and Lemma \ref{L:ext_if_small}),
\begin{equation}\label{extinbd3}
\P\left(\cE_N; V_N(t) < \ep\right)  - \P_\nu(Z_t< \ep) \ge \P\left( V_N(t) < \ep\right)(1-\delta) - \P_\nu(Z_t < \ep).
\end{equation}
The right-hand side converges to $- \delta \P_\nu(Z_t< \ep) \ge - \delta$. Hence for $N$ large enough, it is greater or equal to $-2 \delta$. Putting together \eqref{extinbd2} and \eqref{extinbd3}, we conclude that for $N$ large enough
$$
\left|\P\left(\cE_N; V_N(t)< \ep\right)  - \P_\nu(Z_t< \ep)\right| \le 2 \delta.
$$
Plugging back into \eqref{extinbd4}, we obtain that
$$
|\P(\cE_N) - \P_\nu(\cE)| \le 8 \delta
$$
for $N$ large enough, which proves Proposition \ref{P:lim_prob_ext}.
\end{proof}

\begin{proof}[Proof of Theorem \ref{T:kpplim}]

Recall that $y_N$ is a sequence tending slowly to infinity.  As before, we often drop the subscript $N$ and just write $y$ instead of $y_N$. Denote by $Q^{N}_m (\cE_N)$ the probability of extinction when we start from a configuration of particles whose positions are encoded in an atomic measure $m$,
and let $\E^{N}_m(\cdot)$ denote the corresponding expectation.
We want to compute the probability $Q^{N}_{\delta(L+x)}(\cE_N)= 1 - Q_\mu (L+x)$ of extinction starting with one particle at $L+x$, where $\delta(x)$ denotes the unit Dirac mass at $x$. Let $N_{y}$ be the number of particles that hit $L-y$ when particles are killed at $L-y$.
Then because each of these particles must eventually become extinct for the whole process to die out, we have that
\begin{equation}\label{E:from L+x to L-y}
1 - Q_{\mu}(L+x) = \E^{N}_{\delta(L+x)} \left[Q^{N}_{N_y \delta(L-y)}(\cE_N)\right]. 
\end{equation}
Indeed, for extinction it is irrelevant that these $N_y$ particles reach $L - y$ at different times.

Therefore, we may consider a sequence of random initial configurations consisting of $N_y$ particles at position $L-y$.
Using Lemma \ref{L:approxNy} we see that
\begin{equation}\label{E:Nyx}
N_y  y e^{-\sqrt 2 y} \to_d e^{\sqrt 2 x} W,
\end{equation}
where $W$ is the limiting random variable in \eqref{E:conv to W}.
We claim that this sequence of initial configurations satisfies almost surely the conditions of Proposition  \ref{P:lim_prob_ext} (which are those of Proposition 1 in \cite{bbs}).  This will establish the first part of Theorem \ref{T:kpplim}.  Indeed, we have
\begin{equation}
\frac{Z_N(0)}{N (\log N)^2} =  \frac{ N_y  \sin (\pi y/L) e^{\mu (L-y)} }{N (\log N)^2 }  \to_d \pi \sqrt{2} e^{\sqrt{2} x} W.
\end{equation}
Moreover,
\begin{equation}
\frac{Y_N(0)}{N (\log N)^3} =  \frac{N_y e^{\mu (L-y)} }{N(\log N)^3} = \frac{e^{\mu L}}{N (\log N)^3} \cdot
N_y y e^{-\mu y} \cdot \frac{1}{y} \rightarrow_p 0
\end{equation}
because the first factor converges to 1, the third converges to 0, and the second converges in distribution
to $e^{\sqrt{2} x}W.$
This proves the first half of the theorem.

According to Proposition \ref{P:lim_prob_ext} and Lemma \ref{L:extprobinf}, if we take a sequence of initial configurations such that
$$
\frac{Z_N(0)}{N(\log N)^2} \to c W
$$
in distribution, where $c$ is a constant, and
$Y_N(0)/(N(\log N)^3)$ converges to $0$ in probability, then
\begin{equation}\label{E:proba ext}
\lim_{N \to \infty} Q^{N}_{m}(\cE_N) = \P_{\nu}( \cE) = \E(e^{ -c \alpha W})
\end{equation}
where $m$ is the atomic measure representing the initial configuration of particles that satisfies the above conditions, $\nu$ is the law of $cW$, and $\alpha = \exp (-a/2\pi^2)$.

Therefore, by (\ref{E:from L+x to L-y}), (\ref{E:proba ext}), and (\ref{E:Nyx})  we conclude
$$
\lim_{N\to \infty} Q^{N}_{\delta(L+x)}(\cE_N) = \E \left[  e^{ - \alpha \pi \sqrt{2} e^{\sqrt 2 x} W } \right]= \E \left[  e^{ - e^{\sqrt 2 (x+\beta)} W } \right]
$$
where $\beta = \log (\alpha \pi \sqrt{2}) / \sqrt 2 = (\log(\pi \sqrt{2}) - a/2 \pi^2)/\sqrt{2}$.

It was shown in \cite{nev87} (see also Proposition 24 in \cite{bbs}) that
\begin{equation}\label{LapTW}
\E[e^{-e^{\sqrt{2} u} W}] = \psi(u),
\end{equation}
where $\psi: \R \rightarrow (0, 1)$ solves Kolmogorov's equation
\begin{equation}\label{kolmeq2}
\frac{1}{2} \psi'' - \sqrt{2} \psi' = \psi(1 - \psi).
\end{equation}
Thus
$$
\lim_{N \to \infty} Q^{N}_{\delta(L+x)}(\cE_N) = \psi(x+\beta).$$
Letting $\theta(x) = 1 - \psi(x + \beta)$, we get $$\lim_{N \rightarrow \infty} 1 - Q^{N}_{\delta(L+x)}(\cE_N) = \theta(x),$$
and from (\ref{kolmeq2}) it is easily verified that $\theta$ satisfies \eqref{kolmeq}.
This finishes the proof of Theorem \ref{T:kpplim}.
\end{proof}

\section{Proof of Theorem \ref{T:survivalthm}}

Throughout this section, we consider branching Brownian motion with a drift to the left of $\mu$ started with a single particle at $x$, where $x$ depends on $\epsilon$ and therefore on $N$.
The proof of Theorem \ref{T:survivalthm} will follow directly from the next three lemmas.
Fix $\alpha>0$.
Let $R$ be the total number of particles that hit $L- \alpha$ when particles are killed upon hitting this barrier.  (Note that $\sqrt{2} \alpha$ corresponds to the parameter $A$ in \cite{bbs}.)  The first lemma controls the expectation of this number.

\begin{Lemma} \label{L: E(R)}
For each $\alpha >0$ fixed,
$$\E[R] = \frac{e^{\sqrt 2 \alpha}}{\pi \sqrt 2 \alpha} e^{\mu x} \sin \left( \frac{\pi x}{L-\alpha} \right) \frac1{N (\log N)^2} (1+C_{N})$$
where $C_{N} \to 0$ as $N\to \infty.$
\end{Lemma}
\begin{proof}[Proof of Lemma \ref{L: E(R)}]
Start with a particle at $x$.
By Lemma 15 of \cite{bbs} with $A=\sqrt{2} \alpha$, the ``rate" at which particles are hitting $L-\alpha$ at time $t$, for $t \gg L^2$, is approximately
$$\rho^{(N)}(t) = 2 \pi e^{\sqrt{2} \alpha} e^{(1 - \mu^2/2 - 2 \pi^2/(L-\alpha)^2)t} e^{\mu x} \sin \bigg( \frac{\pi x}{L - \alpha} \bigg) \frac{1}{N (\log N)^5}.$$
More precisely, define $R^{(N)}([t,t+\delta])$ to be the number of particles that reach $L - \alpha$ between times $t$ and $t+\delta$.  Then we have the bounds
\begin{equation}\label{E:R}
\delta \rho^{(N)} (t) (1+E_{N,t}) (1 + C^{(1)}_{N, \delta}) \leq \E [R^{(N)}([t,t+\delta])] \leq \delta \rho^{(N)} (t) (1+E_{N,t}) (1 + C^{(2)}_{N, \delta}),
\end{equation}
where $E_{N,t}$ is the constant defined by equation (16) in \cite{bbs} which satisfies
\begin{equation}\label{E:bound on E}
\vert E_{N,t} \vert \le  \sum_{n=2}^\infty n^2 e^{-\pi^2 (n^2-1) t/2(L - \alpha)^2},
\end{equation}
and $$\lim_{\delta \to 0} \lim_{N \to \infty} C^{(i)}_{N, \delta} = 0$$ for $i = 1, 2$.  Note that it can be seen from the proof of Lemma 15 in \cite{bbs} that the constants $C^{(i)}_{N, \delta}$ do not depend on $t$.

Let $\tau_N = \theta_N (\log N)^2$, where $1 \ll \theta_N \ll \log N$ as $N \rightarrow \infty$.  We have
\begin{align}
\E[R^{(N)}([\tau_N, \infty))] &\geq \sum_{k=0}^{\infty} \E[R^{(N)}([\tau_N + k \delta, \tau_N + (k+1) \delta])] \nonumber \\
&\geq \delta (1 + C_{N, \delta}^{(1)}) \bigg( 1 - \sum_{n=2}^{\infty} n^2 e^{-\pi^2 (n^2 - 1)\tau_N/2(L - \alpha)^2} \bigg) \sum_{k=0}^{\infty} \rho^{(N)}(\tau_N + k \delta) \nonumber \\
&= \delta (1 + C_{N, \delta}^{(1)}) \bigg( 1 - \sum_{n=2}^{\infty} n^2 e^{-\pi^2 (n^2 - 1)\tau_N/2(L - \alpha)^2} \bigg) \frac{\rho^{(N)}(\tau_N)}{1 - e^{(1 - \mu^2/2 - 2 \pi^2/(L - \alpha)^2)\delta}} \nonumber
\end{align}
and likewise
$$\E[R^{(N)}([\tau_N, \infty))] \leq \delta (1 + C_{N, \delta}^{(2)}) \bigg( 1 + \sum_{n=2}^{\infty} n^2 e^{-\pi^2 (n^2 - 1)\tau_N/2(L - \alpha)^2} \bigg) \frac{\rho^{(N)}(\tau_N)}{1 - e^{(1 - \mu^2/2 - 2 \pi^2/(L - \alpha)^2)\delta}}.$$
As $N \rightarrow \infty$, we have (see equation (38) of \cite{bbs}) $$1 - \frac{\mu^2}{2} - \frac{\pi^2}{2(L - \alpha)^2} \sim - \frac{2 \sqrt{2} \pi^2 \alpha}{(\log N)^3}$$ and therefore, since $\tau_N \ll (\log N)^3$,
$$\rho^{(N)}(\tau_N) \sim 2 \pi
e^{\sqrt{2} \alpha} e^{\mu x} \sin \bigg( \frac{\pi x}{L - \alpha} \bigg) \frac{1}{N (\log N)^5}.$$  By letting $N \rightarrow \infty$ and comparing the upper and lower bounds, then taking $\delta \rightarrow 0$, we get
\begin{equation}\label{newERN}
\E[R^{(N)}([\tau_N, \infty))] \sim \frac{e^{\sqrt 2 \alpha}}{\pi \sqrt 2 \alpha} e^{\mu x} \sin \left( \frac{\pi x}{L-\alpha} \right) \frac1{N (\log N)^2}.
\end{equation}

To complete the proof of Lemma \ref{L: E(R)}, it suffices to show that $\E[R^{(N)}([0, \tau_N])]$ is much smaller than the right-hand side of (\ref{newERN}).  We do this by following the argument at the beginning of the proof of Proposition 16 in \cite{bbs}.
Recall from Lemma 6 of \cite{bbs} that $$V(t) = \sum_{i=1}^{M(t)} X_i(t) e^{\mu X_i(t) + (\mu^2/2 - 1)t}$$ defines a martingale for branching Brownian motion with particles killed at the origin.  Furthermore, we see that
because $\mu^2/2 - 1 < 0$, this process remains a supermartingale if particles are stopped, but not killed, when reaching $L - \alpha$.  Therefore, $$V(t) = R^{(N)}([0, t]) (L - \alpha) e^{\mu (L - \alpha) + (\mu^2/2 - 1)t} + \sum_{i=1}^{M(t)} X_i(t) e^{\mu X_i(t) + (\mu^2/2 - 1)t}$$ defines a supermartingale.  It follows that
$$x e^{\mu x} = V(0) \geq \E[V(\tau_N)] \geq \E[R^{(N)}([0, \tau_N])] (L - \alpha) e^{\mu(L - \alpha) + (\mu^2/2 - 1)\tau_N}.$$  Now since $-(\mu^2/2 - 1)(\log N)^2$ is bounded by some constant $C$, we get the bound
\begin{equation}\label{hy2}
\E[R^{(N)}([0, \tau_N])] \leq \frac{ x e^{\mu x}}{L - \alpha} e^{-\mu(L - \alpha)}e^{C\theta_N} \leq \frac{C' x e^{\mu x} e^{\mu \alpha} e^{C \theta_N}}{N(\log N)^4}
\end{equation}
for some other constant $C'$.
This expression will be much smaller than the right-hand side of (\ref{newERN}) as $N \rightarrow \infty$ provided that $$\frac{x e^{C \theta_N}}{(\log N)^2} \ll \sin \bigg( \frac{\pi x}{L - \alpha} \bigg).$$  Because $L - x \gg 1$, this can be arranged by making sure that $\theta_N \rightarrow \infty$ sufficiently slowly.
\end{proof}

The second Lemma tells us that the expectation above is not dominated by a very small probability of a large number of particles reaching $L-\alpha.$
\begin{Lemma} \label{L:E(R /R>0)}
Assume $\alpha \geq 1$.
There exists $C>0$
not depending on $\alpha$ and a sequence $C_{N,\alpha}$ tending to zero as $N \rightarrow \infty$ for each fixed $\alpha$ such that
$$\E(R^2)\le (C + C_{N, \alpha}) \E(R).$$
\end{Lemma}

\begin{proof}[Proof of Lemma \ref{L:E(R /R>0)}]

We follow the proof of Proposition 18 in \cite{bbs}.
Here we have the slightly simpler decomposition $R^2 =R+2 Y^x$ where $Y^x$ is the number of distinct pairs of
particles that ever hit $L-\alpha.$

Define $q_t(x,y)$ so that if there is initially one particle at $x$, the expected number of particles in a set $B \subset (0, L - \alpha)$ at time $t$ is given by $\int_B q_t(x,y) \: dy$.  Let $h(y)$ be the expected number of offspring of a single particle at $y$ that
hit $L - \alpha$.  Note that if at time $t$ a branching event causes a particle at $y$ to split into two particles, then there will be on average $h(y)^2$ distinct pairs of particles that hit $L - \alpha$ and have their most recent common ancestor at time $t$.  It follows that
\begin{align}
\E[Y^x] &= \int_0^{\infty} \int_0^{L - \alpha} q_t(x,y) h(y)^2 \: dy \: dt \nonumber \\
&= \int_0^{(\log N)^2} \int_0^{L - \alpha} q_t(x,y) h(y)^2 \: dy \: dt + \int_{(\log N)^2}^{\infty} \int_0^{L - \alpha} q_t(x,y) h(y)^2 \: dy \: dt  \nonumber \\
&= T_1 + T_2. \nonumber
\end{align}
To bound $h(y)$, we use (\ref{newERN}) and (\ref{hy2}) with $y$ in place of $x$ and $\theta_N = 1$ for all $N$.  In this case, $E_{N,t}$ is still bounded by a constant, and we get
\begin{equation}\label{hyeq}
h(y) \leq \frac{C}{N (\log N)^2} \cdot \frac{e^{\sqrt{2} \alpha}}{\pi \sqrt{2} \alpha} \cdot e^{\mu y} \sin \bigg( \frac{\pi y}{L - \alpha} \bigg) + \frac{C y e^{\mu y} e^{\mu \alpha}}{N (\log N)^4},
\end{equation}
where here and throughout the proof, the value of the constant $C$ may change from line to line.
Define $v_t(x,y)$ so that if a Brownian particle (without drift or branching) is started at $x$ and is killed upon reaching $0$ or $L - \alpha$, then the probability that the particle is in a set $B \subset (0, L - \alpha)$ at time $t$ is given by $\int_B v_t(x,y) \: dy$.
To bound $T_1$, we use the fact (see equation (28) of \cite{bbs}) that $q_t(x,y) \leq C e^{\mu(x - y)} v_t(x,y)$ for $t \leq (\log N)^2$ to get
\begin{align}
T_1 &\leq \int_0^{(\log N)^2} \int_0^{L - \alpha} e^{\mu(x - y)} v_t(x,y) \bigg( \frac{C}{N (\log N)^2} \cdot \frac{e^{\sqrt{2} \alpha}}{\pi \sqrt{2} \alpha} \cdot e^{\mu y} \sin \bigg( \frac{\pi y}{L - \alpha} \bigg) + \frac{C y e^{\mu y} e^{\mu \alpha}}{N (\log N)^4} \bigg)^2 \: dy \: dt \nonumber \\
&\leq C e^{\mu x} \int_0^{L - \alpha} \bigg( \int_0^{(\log N)^2} v_t(x,y) \: dt \bigg) \frac{e^{\mu y} e^{2 \sqrt{2} \alpha}}{N^2 (\log N)^4} \bigg( \frac{1}{\alpha^2} \sin \bigg( \frac{\pi y}{L - \alpha} \bigg)^2 + \frac{y^2}{(\log N)^4} \bigg) \: dy. \nonumber
\end{align}
Now using that $\int_0^{\infty} v_t(x,y) \: dt \leq 2x(L - \alpha - y)/(L - \alpha)$ from standard Green's function results for Brownian motion (see Section 2 of \cite{bbs}), we get
\begin{align}
T_1 &\leq \frac{C e^{\mu x} e^{2 \sqrt{2} \alpha}}{N^2 (\log N)^4} \int_0^{L - \alpha} \frac{x(L - \alpha - y)}{L - \alpha} e^{\mu y} \bigg( \frac{1}{\alpha^2} \sin \bigg( \frac{\pi y}{L - \alpha} \bigg)^2 + \frac{y^2}{(\log N)^4} \bigg) \: dy. \nonumber
\end{align}
The integral is dominated by values of $y$ near $L - \alpha$, for which the term in parentheses contributes $1/(\log N)^2$ and $(L - \alpha - y)/(L - \alpha)$ contributes another $1/(\log N)$.  Therefore, the integral is of the order $x e^{\mu(L - \alpha)}/(\log N)^3 \leq
C x N e^{-\sqrt{2} \alpha}$, and we get
\begin{equation}\label{T1}
T_1 \leq \frac{Cx e^{\mu x} e^{\sqrt{2} \alpha}}{N (\log N)^4}.
\end{equation}

It remains to bound $T_2$.  For this we use that for $t \geq (\log N)^2$, we have $$q_t(x,y) \leq \frac{C}{L} e^{-2 \sqrt{2} \pi^2 \alpha t/(\log N)^3} e^{\mu x} \sin \bigg( \frac{\pi x}{L - \alpha} \bigg) e^{-\mu y} \sin \bigg( \frac{\pi y}{L - \alpha} \bigg)$$
(see Lemma 5 and equation (38) of \cite{bbs}).
We write $T_2 \leq C(T_{2,a} + T_{2,b})$, where $T_{2,a}$ comes from bounding $h(y)^2$ by the square of the first term on the right-hand side of (\ref{hyeq}) and $T_{2,b}$ comes from bounding $h(y)^2$ by the square of the second term on the right-hand side of (\ref{hyeq}).  Now
\begin{align}\label{T1b}
T_{2,a} &\leq \frac{C}{N^2 (\log N)^4} \bigg( \frac{e^{\sqrt{2} \alpha}}{\pi \sqrt{2} \alpha} \bigg)^2 \cdot \frac{1}{L} e^{\mu x} \sin \bigg( \frac{\pi x}{L - \alpha} \bigg) \nonumber \\
&\hspace{.5in}\times \int_0^{\infty} \int_0^{L - \alpha} e^{\mu y} \sin \bigg( \frac{\pi y}{L - \alpha} \bigg)^3 e^{-2 \sqrt{2} \pi^2 \alpha t/(\log N)^3} \: dy \: dt \nonumber \\
&\leq \frac{C}{N^2 (\log N)^5} \bigg( \frac{e^{\sqrt{2} \alpha}}{\pi \sqrt{2} \alpha} \bigg)^2 e^{\mu x} \sin \bigg( \frac{\pi x}{L - \alpha} \bigg) \cdot \frac{(\log N)^3}{\alpha} \cdot
\frac{e^{\mu (L - \alpha)}}{(\log N)^3} \nonumber \\
&\leq \frac{C}{N (\log N)^2} \cdot \frac{e^{\sqrt{2} \alpha}}{\alpha^3} \cdot e^{\mu x} \sin \bigg( \frac{\pi x}{L - \alpha} \bigg).
\end{align}
and
\begin{align}\label{T2b}
T_{2,b} &\leq \frac{C}{N^2 (\log N)^8} e^{2 \mu \alpha} \cdot \frac{1}{L} e^{\mu x} \sin \bigg( \frac{\pi x}{L - \alpha} \bigg) \nonumber \\
&\hspace{.5in}\times \int_0^{\infty} \int_0^{L - \alpha} e^{\mu y} y^2 \sin \bigg( \frac{\pi y}{L - \alpha} \bigg) e^{-2 \sqrt{2} \pi^2 \alpha t/(\log N)^3} \: dy \: dt \nonumber \\
&\leq \frac{C}{N^2 (\log N)^9} e^{2 \mu \alpha} \cdot e^{\mu x} \sin \bigg( \frac{\pi x}{L - \alpha} \bigg) \cdot \frac{(\log N)^3}{\alpha} \cdot \frac{N (\log N)^3 \cdot (\log N)^2 e^{-\mu \alpha}}{\log N} \nonumber \\
&= \frac{C}{N (\log N)^2} \cdot \frac{e^{\mu \alpha}}{\alpha} \cdot e^{\mu x} \sin \bigg( \frac{\pi x}{L - \alpha} \bigg).
\end{align}

Recall that Lemma \ref{L: E(R)} gives the expected number of particles that reach $L - \alpha$.  By comparing this expectation with the bounds in (\ref{T1b}), and (\ref{T2b}), we get that $T_2 \leq C \E[R]$.  Note that all the constants in the bounds of $T_1$ and $T_2$ above are independent of $\alpha$ over the range $\alpha \geq 1$.  Furthermore, because $L - x \gg 1$, it follows from (\ref{T1}) that $T_1/\E[R] \rightarrow 0$ as $N \rightarrow \infty$ for each fixed $\alpha$.  These observations imply the statement of the Lemma.
\end{proof}

When $\alpha$ is large, the probability that an individual particle at $L - \alpha$ has descendants that survive forever is small, meaning the number of such particles will likely either be $0$ or $1$ and the bound from Markov's Inequality should be precise.
Therefore we have the following Lemma.

\begin{Lemma} \label{L: Qmu(x) = E(R) Qmu(L-alpha)}
We have $$Q_{\mu}(x) = \frac{e^{\sqrt{2} \alpha}}{\pi \sqrt{2} \alpha} \cdot e^{\mu x} \sin \bigg( \frac{\pi x}{L} \bigg) \cdot \frac{1}{N (\log N)^2} \cdot Q_{\mu}(L - \alpha) (1 + C_{N,\alpha} + o(\alpha^{-1})),$$
where 
$o(\alpha^{-1})$ denotes a term that tends to zero as $\alpha \rightarrow \infty$ and does not depend on $N$ while for each fixed $\alpha, C_{N,\alpha} \to 0$ as $N\to \infty.$
\end{Lemma}

\begin{proof}[Proof of Lemma \ref{L: Qmu(x) = E(R) Qmu(L-alpha)}]

Consider the particles that reach $L - \alpha$ when particles are stopped upon reaching this level, and let $S$ be the number of these particles that have an infinite line of descent.  Since each particle that reaches $L - \alpha$ has probability $Q_\mu(L-\alpha)$ of having an infinite line of descent, we have $\E(S)=\E(R)Q_\mu(L-\alpha)$, and more generally, conditional on $R$ the random variable $S$ is binomial $(R, q)$ with $q= Q_\mu(L-\alpha).$
It follows that
$$0 \le  \E(S) - \P(S>0) = \E(Rq -(1-(1-q)^R)) \le \E(R^2q^2) \le (C + C_{N, \alpha}) q^2 \E(R) = (C + C_{N, \alpha})q\E(S),$$
where we have used $Rq -(1-(1-q)^R) \le R^2q^2$ for the second inequality and Lemma \ref{L:E(R /R>0)} for the third inequality.  It follows that
$$
 \E(S) (1 - (C + C_{N, \alpha})q) \leq \P(S > 0) \leq \E(S).
$$
Because $q \rightarrow 0$ as $\alpha \rightarrow \infty$, we have
\begin{equation}\label{Sprob}
\P(S>0) = \E(R) q (1+ C_{N, \alpha} + o(\alpha^{-1})).
\end{equation}
The martingale argument in the proof of Lemma \ref{L:ext_if_small} implies that almost surely the process survives if and only if $S > 0$.  Consequently, equation (\ref{Sprob}) and Lemma \ref{L: E(R)} imply the result.
\end{proof}

\begin{proof}[Proof of Theorem \ref{T:survivalthm}.]
We now use Lemma \ref{L: Qmu(x) = E(R) Qmu(L-alpha)} to prove Theorem \ref{T:survivalthm}. Indeed letting $N \rightarrow \infty$ and using Theorem \ref{T:kpplim}, we get that  $$Q_{\mu}(x) \sim \frac{e^{\sqrt{2} \alpha} \theta(-\alpha)}{\pi \sqrt{2} \alpha} \cdot e^{\mu x} \sin \bigg( \frac{\pi x}{L} \bigg) \cdot \frac{1}{N (\log N)^2} \cdot (1 + o(\alpha^{-1})).$$  Since the left-hand side does not depend on $\alpha$, it follows that as $\alpha \rightarrow \infty$, the expression $\alpha^{-1} e^{\sqrt{2} \alpha} \theta(-\alpha)$ must tend to a limiting constant, which implies the result.
\end{proof}

\section{Proof of Proposition \ref{P:ic}}

In order to simplify the proof we choose $\alpha = 0$ (all arguments are easily adapted for a generic value of $\alpha$). Let $a>0$ be arbitrary and let $T= a (\log N)^2 $. Since $\alpha = 0$, we are assuming that there is initially one particle at $x=L$. Fix $y\ge 0$, and let $N_y$ be the number of particles ever touching $L-y$, if particles are killed upon touching this level.
Let $t_1\le \dots \le t_{N_y}$ denote the respective times at which these particles hit $L-y$. Let $\cF$ be the $\sigma$-field generated by $N_y$ and $t_1, \ldots, t_{N_y}$.
Observe that, conditionally on $\cF$, on the event $\cV$ that $t_{N_y} \le T$, we have that $Z_N(T)$ and $Y_N(T)$ are the sum of $N_y$ independent random variables.
More precisely, for $1\le i \le N_y$, let $(X^i_j(t))_{1\le j \le M^i(t)}$ denote the positions of the descendants of the $i$th particle hitting $L-y$, at time $t + t_i$. Let
\begin{equation}\label{Zi}
Z^i(t) = \sum_{j=1}^{M^i(t)} e^{\mu X^i_j(t)}\sin\left(\frac{\pi X^i_j(t)}L\right) \indic{X^i_j(t)\le L}
\end{equation}
and likewise, let
$$
Y^i(t) = \sum_{j=1}^{M^i(t)} e^{\mu X^i_j(t)}.
$$
Then we may write, on the event $\cV$,
$$
Z_N(T) = \sum_{i=1}^{N_y} Z^i(T-t_i), \  \ Y_N(T) = \sum_{i=1}^{N_y} Y^i(T-t_i),
$$
and observe that, given $\cF$, the summands are independent (but not identically distributed) random variables. Our first lemma shows that between times $t_{N_y}$ and $T$, with high probability no particle hits $L$.

Reasoning as in the proof of Lemma \ref{L:approxNy}, we may choose $y= y_N$ in such a way that $y_N \to \infty$ but  slowly enough that
\begin{equation}\label{E:Tysmall}
t_{N_y} / (\log N)^2 \to 0
\end{equation}
in probability.

\begin{Lemma}\label{newRlem}
\label{L:hitL} Let $R$ denote the number of particles touching $L$ between times $t_{N_y}$ and $T$. Then $\P(R\ge 1) \to 0$.
\end{Lemma}

\begin{proof}
This follows from the same sort of arguments as in Lemma \ref{L:Prop1618}. More precisely, for $1\le i \le N_y$ let $R^i(t)$ denote the number of descendants of the $i$th particle hitting $L-y$ that touch $L$ between times $t_i$ and $t_i +t$ if particles are killed when touching $L$. Then $R = \sum_{i=1}^{N_y} R^i(T-t_i) \le \sum_{i=1}^{N_y} R^i(T)$. Note that, conditionally given $\cF$, the random variables $R^i(T)$, $1\le i \le N_y$, are independent and identically distributed. 
Reasoning as in (68) in \cite{bbs} (or simply by using directly (68) if $a =1$), we get
\begin{equation*}
\E(R|\cF) \le C \frac{\sum_{i=1}^{N_y} Y^i(0)}{N(\log N)^3} = C N_y  \frac{e^{\mu (L-y)} }{N(\log N)^3} \le C N_y e^{-\mu y }.
\end{equation*}
By Lemma \ref{L:approxNy}, there exists a random variable $W$ such that $ye^{-\sqrt{2} y }N_y \to W$ in distribution. Fix $\delta >0$. Choose $K\ge 0$ large enough that $\P(W>K)\le \delta$. Thus
for $N$ large enough, $\P(\cV_0)\ge 1- 2\delta$ where $\cV_0 = \{ N_y \le K e^{\sqrt{2} y} /y\}$. On $\cV_0$, we see that
$$
\E(R|\cF) \le  C K y^{-1}.
$$
Thus by Markov's inequality, $\P(R\ge 1) \le \P(\cV_0^\complement) + CKy^{-1} \le 3 \delta$ for $N$ sufficiently large. This proves Lemma \ref{L:hitL}.
\end{proof}

Let $\cV_1 = \{R=0\}$. Then on $\cV_1$, we may identify $Z^i(t)$ for each $1\le i \le N_y$ with the random variable $\hat Z^i(t)$, where $\hat Z^i(t)$ is defined as $Z^i(t)$ in \eqref{Zi}  except that the sum is only over those particles whose ancestors between times $t_i$ and $t_i + t$ never hit $L$. Likewise, $Y^i(t)$ is equal to $\hat Y^i(t)$, where $\hat Y^i$ is defined in the analogous fashion.
Let $\hat Z_N(T) = \sum_{1\le i \le N_y} \hat Z^i(T-t_i)$, and define $\hat Y_N(T)$ analogously. Observe that, conditionally given $\cF$,  $\hat Z^i(t)$ is a martingale for each $1\le i \le N_y$. Thus we deduce that
\begin{equation}
\label{E:hatZind_exp}
\E(\hat Z^i(T-t_i)|\cF) =  Z^i(0) = e^{\mu (L-y)} \sin\left(\frac{\pi(L-y)}L\right).
\end{equation}
Moreover,
$$
\Var (\hat Z^i(T-t_i) |\cF) \le \E((\hat Z^i(T-t_i))^2|\cF),
$$
so using Lemma 9 in \cite{bbs}, we deduce, on the event $ \cV_2= \{t_{N_y} \le (a/2) (\log N)^2\}$, that
$$
\Var (\hat Z^i(T-t_i) |\cF) \le C e^{\mu (L-y)} e^{\mu L} \frac1{L^2} \le C N^2 L^4 e^{-\mu y}.
$$
Since the variables $\hat Z^i(T-t_i)$ are conditionally independent given $\cF$, we obtain that on $\cV_2$,
$$
\Var( \hat Z_N(T)|\cF) \le  C N_y N^2 L^4 e^{-\mu y}.
$$
Fix $\eta>0$. Let $\cV_\eta = \cV_2 \cap \{ W(1-\eta) \le ye^{-\mu y} N_y \le W(1+\eta) \}$.
Then on $\cV_\eta$,
we get that
\begin{equation}\label{E:Zexpvar}
\Var(\hat Z_N(T) |\cF) \le C W  \frac{N^2 L^4}y;  \ \ \E(\hat Z_N(T) |\cF)  =  \pi \sqrt{2} W N (\log N)^2 q,
\end{equation}
where $q \in (1-\eta, 1+\eta)$ almost surely on $\cV_\eta$. Since $y_N \to \infty$ and $\P(\cV_\eta) \to 1$ (due to (\ref{E:Tysmall})) it is now easy to deduce from Chebyshev's inequality that
\begin{equation*}
\frac{\hat Z_N(T)}{N( \log N)^2} \to \pi \sqrt{2} W
\end{equation*}
in distribution.  Because $\P(\cV_1) \rightarrow 1$ by Lemma \ref{newRlem} and $Z_N(T) = \hat Z_N(T)$ on $\cV_1$, it follows that
\begin{equation}\label{E:icZN}
\frac{Z_N(T)}{N( \log N)^2} \to \pi \sqrt{2} W
\end{equation}
in distribution.

Likewise, using (18) from \cite{bbs}, we see that on the event $\cV_2$, since $t_{N_y} \le (a/2) (\log N)^2$, so that $T-t_i \ge (a/2) (\log N)^2$ for each $1\le i\le N_y$ (and thus the term $E_2$ in (18) of \cite{bbs} is bounded),
$$
\E(\hat Y_N(T) | \cF) \le C N_y e^{\mu (L-y)} \sin\left(\frac{\pi(L-y)}L\right).
$$
and therefore we deduce that
\begin{equation}\label{E:icYN}
\frac{Y_N(T)}{N(\log N)^3} \to 0
\end{equation}
in probability. Since $N(\log N)^3 = \exp({\sqrt{2} L}) = \exp({\sqrt{2}\pi \eps^{-1/2}})$, this proves (\ref{E:icY}). Likewise, (\ref{E:icZ}) follows from (\ref{E:icZN}) since $N(\log N)^2  \sim \sqrt{2} L \exp({\sqrt{2} L}) = \pi \sqrt{2} \eps^{-1/2}  \exp({\sqrt{2}\pi \eps^{-1/2}})$. The proof of Proposition \ref{P:ic} is finished.

\bigskip
\noindent {\bf {\Large Acknowledgments}}

\bigskip
\noindent The authors thank Damien Simon for helpful discussions.  They also thank Elie Aidekon, Simon Harris, and Robin Pemantle for providing them with a preliminary version of the reference \cite{ahp}.

\end{document}